\documentclass{amsart}

\usepackage{esint}
\usepackage{amsfonts}
\usepackage{amssymb}
\usepackage{xcolor}
\usepackage{url, hyperref}

\newtheorem{theorem}{Theorem}[section]

\newtheorem{corollary}[theorem]{Corollary}

\newtheorem{definition}[theorem]{Definition}

\newtheorem{lemma}[theorem]{Lemma}

\newtheorem{question}[theorem]{Question}

\numberwithin{equation}{section}

\begin{document}
\title[Three-manifolds with positive scalar curvature]{Geometry of
three-dimensional manifolds with positive scalar curvature}
\author{Ovidiu Munteanu and Jiaping Wang}

\begin{abstract}
The purpose of this paper is to derive volume and other geometric
information for three-dimensional complete manifolds with positive scalar
curvature. In the case that the Ricci curvature is nonnegative, it is shown
that the volume of the manifold must be of linear growth when the scalar
curvature is bounded from below by a positive constant. This answers a
question of Gromov in the affirmative for dimension three. Volume growth
estimates are also obtained for the case when scalar curvature decays to
zero. In fact, results of similar nature are established for the more
general case that the Ricci curvature is asymptotically nonnegative.
\end{abstract}

\address{Department of Mathematics, University of Connecticut, Storrs, CT
06268, USA}
\email{ovidiu.munteanu@uconn.edu}
\address{School of Mathematics, University of Minnesota, Minneapolis, MN
55455, USA}
\email{jiaping@math.umn.edu}
\maketitle

\section{Introduction}

One of our purposes in this paper is to provide an affirmative answer
in the case of dimension $n=3$ to the
following question posed by Gromov \cite{G}.

\begin{question}
Let $(M^{n},g)$ be a complete manifold with nonnegative Ricci curvature. If
its scalar curvature $S\geq 1,$ is there a positive constant $C$ such that 
$\mathrm{V}_{p}(R)\leq C\,R^{n-2}$ for all $R>0,$ where $\mathrm{V}_{p}(R)$
is the volume of the geodesic ball $B_{p}(R)$ centered at a point $p$ and of
radius $R$?
\end{question}

In fact, Yau \cite{Y} asked even more ambitiously whether

\begin{equation*}
\limsup_{R\rightarrow \infty }R^{2-n}\,\int_{B_{p}(R)}S<\infty
\end{equation*}%
on a complete manifold $(M^{n},g)$ with nonnegative Ricci curvature.

In this paper we establish the following result.

\begin{theorem}
\label{A5} Let $(M^3,g)$ be a complete three-dimensional manifold with
non-negative Ricci curvature.

\begin{itemize}
\item If the scalar curvature is bounded below by $S\geq 1$ on $M,$ then
there exists a universal constant $C>0$ such that 
\begin{equation*}
\mathrm{V}_{p}\left( R\right) \leq C\,R
\end{equation*}%
for all $R>0$ and $p\in M.$

\item If the scalar curvature is bounded below by 
\begin{equation*}
S\left( x\right) \geq \frac{C_0}{r^{\alpha }\left( x\right) +1}
\end{equation*}%
for some $\alpha \in \left[ 0,1\right] $ and $C_0>0$, where $r(x)$ is the geodesic
distance from $x$ to a fixed point $p\in M,$ then%
\begin{equation*}
\mathrm{V}_{p}\left( R_{i}\right) \leq \frac{C}{C_0}\,R_{i}^{\alpha +1}
\end{equation*}%
for a sequence $R_{i}\rightarrow \infty$, where $C$ is a universal constant. 
\end{itemize}
\end{theorem}

In fact, we consider the more general case that the Ricci curvature is only
assumed to be asymptotically nonnegative, namely, the Ricci curvature satisfies 
\begin{equation}\label{Ricci_bd_k}
\mathrm{Ric}(x)\geq -k\left(r(x)\right),
\end{equation}
 for a nonnegative non-increasing function $
k(r)$ with $$\int_{0}^{\infty} r\,k(r)\,dr<\infty.$$

\begin{theorem}
\label{A2} Let $(M^3,g)$ be a complete three-dimensional manifold with finite
first Betti number and finitely many ends. Suppose that its Ricci curvature
is asymptotically nonnegative. If the scalar curvature $S$ is bounded below
by a positive constant, then $M$ is parabolic, i.e., it does not admit any
positive Green's function.
\end{theorem}

Our next result relates the lower bound of the scalar curvature with the
volume of unit balls.

\begin{theorem}
\label{A3} Let $(M^3,g)$ be a complete three-dimensional manifold with finite
first Betti number and finitely many ends. Suppose that its Ricci curvature
is asymptotically nonnegative. Then there exists a constant $C$ depending only on the function 
$k$ in (\ref{Ricci_bd_k}) such that the scalar curvature $S$ of $M$ satisfies
\begin{equation*}
\liminf_{x\rightarrow \infty }S\left( x\right) \leq \frac{C}{\mathrm{V}
_{p}\left( 1\right) },
\end{equation*}%
where $\mathrm{V}_{p}\left( 1\right) $ denotes the volume of the unit ball 
$B_{p}\left( 1\right) .$
\end{theorem}

It should be pointed out that the assumption of the Ricci curvature being
asymptotically nonnegative is necessary for Theorem \ref{A3} to hold.
Indeed, consider a complete three-dimensional manifold $\left( M,g\right) $
which is isometric to $\left[ 2,\infty \right) \times \mathbb{S}^{2}$ with
the warped product metric $ds_{M}^{2}=dt^{2}+\frac{1}{\ln t}ds_{\mathbb{S}%
^{2}}^{2}$ outside a compact set. Then its scalar curvature $S\left(
x\right) \rightarrow \infty $ as $x\rightarrow \infty $. Note that its Ricci
curvature satisfies $\mathrm{Ric}(x)\geq -k\left( r\left( x\right) \right) $
for $k\left( r\right) =\frac{C}{r^{2}\ln r}$ when $r$ is sufficiently large.
We also remark that the upper bound estimate of scalar curvature $S$ in
terms of the volume $\mathrm{V}_{p}(1)$ is sharp by considering the example
of cylinder $\mathbb{R}\times \mathbb{S}^{2}(r),$ where $\mathbb{S}^{2}(r)$
is the sphere of radius $r$ in Euclidean space $\mathbb{R}^{3}$.

We now indicate some of the ideas involved in the proofs. The proof of
Theorem \ref{A5} uses a variant of the following monotonicity result in \cite%
{MW}.

\begin{theorem}
\label{A4} Let $\left( M^3,g\right) $ be a complete noncompact
three-dimensional manifold with nonnegative scalar curvature. Assume that $M$
has one end and its first Betti number $b_{1}\left( M\right) =0.$ If $M$ is
nonparabolic and the minimal positive Green's function $G\left( x\right)
=G\left( p,x\right) $ satisfies $\lim_{x\rightarrow \infty }G(x)=0,$ then%
\begin{equation*}
\frac{d}{dt}\left( \frac{1}{t}\int_{\{G=t\} }\left\vert \nabla
G\right\vert ^{2}-4\pi t\right) \leq 0
\end{equation*}%
for all $t>0.$ Moreover, equality holds for some $T>0$ if and only if the
super level set $\left\{ G>T\right\} $ is isometric to a ball in
the Euclidean space $\mathbb{R}^{3}.$
\end{theorem}

This type of result has been applied in \cite{AMO} to reprove the positive
mass theorem in \cite{SY2}. Recently, using a variant of the above theorem,
Chodosh and Li \cite{CL} have affirmed a conjecture of Schoen that a stable
minimal hypersurface in Euclidean space $\mathbb{R}^4$ must be flat.
Historically, more general versions of monotonicity formulas were
established by Colding \cite{Co} and Colding-Minicozzi \cite{CM1} for $n$%
-dimensional manifolds with nonnegative Ricci curvature and applied to the
study of uniqueness of the tangent cones for Ricci flat manifolds with
Euclidean volume growth \cite{CM2}. We refer the readers to \cite{CM} for an
exposition on monotonicity formulas in geometric analysis, and \cite{AFM}
for their applications to Willmore type inequalities.

The proof of Theorem \ref{A4} relies on a crucial observation that on a
three-dimensional manifold
\begin{equation}
\mathrm{Ric}\left( \nabla u,\nabla u\right) \left\vert \nabla u\right\vert
^{-2}=\frac{1}{2}S-\frac{1}{2}S_{t }+\frac{1}{\left\vert
\nabla u\right\vert ^{2}}\left( \left\vert \nabla \left\vert \nabla
u\right\vert \right\vert ^{2}-\frac{1}{2}\left\vert \nabla
^{2}u\right\vert^{2}\right)  \label{a1}
\end{equation}
for any harmonic function $u,$ where $S_{t }$ is the scalar
curvature of the level set $l\left(t\right)$ of $u.$ The proof then proceeds
by applying (\ref{a1}) to the Green's function $G$ and integrating the following
Bochner formula over the level sets $l(t)$ of $G.$
\begin{equation*}
\Delta \left\vert \nabla G\right\vert =\left( \left\vert G_{ij}\right\vert
^{2}-\left\vert \nabla \left\vert \nabla G\right\vert \right\vert
^{2}\right) \left\vert \nabla G\right\vert ^{-1}+\mathrm{Ric}\left( \nabla
G,\nabla G\right) \left\vert \nabla G\right\vert ^{-1}.
\end{equation*}%
The term $S_{t }$ is handled with the help of the
Gauss-Bonnet theorem as the assumptions ensure $l(t)$ is compact and
connected.

The idea of rewriting the Ricci curvature term as (\ref{a1}) has origin in
Schoen and Yau \cite{SY}, where they made the important observation that on
a minimal surface $N$ in a three-dimensional manifold $M,$
\begin{equation}
\mathrm{Ric}\left( \nu ,\nu \right) =\frac{1}{2}S-\frac{1}{2}S_{N}-\frac{1}{2%
}\left\vert A\right\vert ^{2}  \label{a2}
\end{equation}%
with $\nu ,$ $S_{N}$ and $A$ being the unit normal vector, the scalar
curvature and the second fundamental form of $N,$ respectively. This
observation enabled them to classify compact stable minimal surfaces in a
three-dimensional manifold with nonnegative scalar curvature. In fact, it
can be used to reprove the well-known classification by Fisher-Colbrie and
Schoen \cite{FS} for complete stable minimal surfaces as well (see \cite{LW4}).
 More generally, an identity of the nature (\ref{a2}) was derived for any
surface $N,$ not necessarily minimal, by Jezierski and Kijowski in \cite{JK}.
 It was applied to level sets of suitably chosen functions to prove both
positive energy and positive mass results in \cite{JK, J}. Recently, this
identity was rediscovered by Stern \cite{S} for the level sets of harmonic
functions. In \cite{BKHS} it has been subsequently used to reprove the
positive mass theorem of Schoen and Yau \cite{SY2}.

The observation we make here is that Theorem \ref{A4} can be refined and
localized to an end of $M.$ Applying it to the so-called barrier functions
of ends then yields Theorem \ref{A5}. Recall by Li and Tam \cite{LT1} that
each end $E$ of a complete manifold admits a barrier function, namely, a
positive harmonic function $u$ satisfying $u=0$ on the boundary of $E.$ Such 
$u$ is bounded if and only if $E$ is nonparabolic. On the other hand, by a
result of Nakai \cite{N}, $u$ can be chosen to be proper in the case $E$ is
parabolic.

The proofs of both Theorem \ref{A2} and Theorem \ref{A3} are very much in
the same spirit of Theorem \ref{A5}. However, the technical details differ.
Again, we work with a barrier function $f.$ For Theorem \ref{A2}, the
difficulty here is that $f$ may not be proper and the Gauss-Bonnet formula
can not be applied directly to the level sets of $f.$ To remedy this, we
consider $u=f\,\psi $ instead, where $\psi $ is a smooth cut-off function on 
$M.$ While this guarantees that the positive level sets of $u$ are compact,
the price we pay is that the function $u$ is no longer harmonic. This
creates many new technical issues. For example, it becomes unclear how to
control the number of connected components of the level sets of $u.$ To get
around the issue, we consider only the component $L(t,\infty )$ of the super
level set $\{u>t\}$ with the fixed point $p\in L(t,\infty ).$ It turns out
that for each end $E$ of $M,$ the unbounded component of $E\setminus
L(t,\infty )$ has exactly one component of the boundary of $L(t,\infty ).$
This fact more or less suffices for our purpose, though there could be other
components of the boundary of $L(t,\infty ).$ Another issue is that the
Bochner formula for $u$ introduces many extra terms. Fortunately, those
terms can be controlled by a judicious choice of the cut-off function $\psi
. $ The choice is based on a result from \cite{BS}.

We refer to \cite{Na, X, Z} for some of the progress on the aforementioned
questions. In particular, it should be mentioned that Theorem \ref{A5} was
proved by Zhu \cite{Z} with a different approach under the additional
assumption that all unit balls of $M$ have a fixed amount of volume.

The structure of the paper is as follows. In Section \ref{sect2}, we present
a proof of Theorem \ref{A5}. Although more general results are proved later
on, we choose to do so as the proof is much less complicated and seems to
better illuminate the main ideas. Section \ref{sect3'} is devoted to the
proof of Theorem \ref{A2} and Section \ref{sect4} to Theorem \ref{A3}.

\textbf{Acknowledgment:} We wish to thank Otis Chodosh for his interest in
our work. We would also like to thank Florian Johne for his careful reading
of the previous version and for his helpful comments. The first author was
partially supported by NSF grant DMS-1811845.

\section{Proof of Theorem \ref{A5} \label{sect2}}

We first establish a localized
version of Theorem \ref{A4}, see Lemma \ref{M}, which moreover does not require curvature bounds. 

 Recall that an end $E$ with respect to a smooth
connected bounded domain $D$ is simply an unbounded component of $M\setminus
D$. According to \cite{LT1}, each end $E$ of a complete manifold carries a
positive harmonic function $u$, which 
is either bounded and satisfies 
\begin{equation}
u=1\text{ on }\partial E\text{, with }\liminf_{x\rightarrow \infty }u(x)=0,
\label{n}
\end{equation}
or 
\begin{equation}
u=0\text{ on }\partial E\text{ with }\limsup_{x\rightarrow \infty
}u(x)=\infty .  \label{p}
\end{equation}
In the former case, $E$ is called nonparabolic and \ in the the latter case,
parabolic. 
\begin{definition}\label{barrier_defn}
The positive harmonic function satisfying either (\ref{n}) or (\ref{p}) is called a barrier 
function of the end $E$. 
\end{definition}
The barrier function in the nonparabolic case has finite energy $%
\int_{M}\left\vert \nabla u\right\vert ^{2}<\infty $. A result of Nakai \cite%
{N} further implies that $u$ can be chosen to be proper whenever $E$ is
parabolic. Let us also note that in both cases, $$\int_{\partial E}\frac{
\partial u}{\partial \nu }\neq 0.$$

In the following, when $\left( M,g\right) $ has finitely many ends and
finite first Betti number, the domain $D$ is assumed to be sufficiently
large such that all representatives of $H_{1}\left( M\right) $ lie in $D$
and $M\setminus D$ has the maximal number of ends. For a harmonic function $%
u $ on an end $E,$ denote by
\begin{eqnarray*}
L\left( a, b\right) &=&\left\{ x\in E:a<u\left( x\right) <b\right\}\\
% \label{level_subset}\\
l\left( t\right) &=&\left\{ x\in E:u\left( x\right) =t\right\}.
%\label{level_set}
\end{eqnarray*}
The following result is essentially contained in \cite{MW}. See also Lemma 2.3 in \cite{LT2}.

\begin{lemma}
\label{MV} Let $\left( M^n,g\right) $ be a complete $n$-dimensional manifold with finite first
Betti number and finite number of ends. For a proper harmonic function $u$
on an end $E,$ its level set $l\left( t\right)$ is connected for all $t.$
\end{lemma}

We now derive a refined and localized version of Theorem \ref{A4}. For any
regular value $t$ of $u$, let
\begin{equation}
w\left( t\right) =\int_{l\left( t\right) }\left\vert \nabla u\right\vert
^{2}.  \label{a3}
\end{equation}
For $\alpha \in \mathbb{R},$ define
\begin{eqnarray}
H_{\alpha }\left( t\right) &=&t^{\alpha }\frac{dw}{dt}\left( t\right)
-\left( \alpha +3\right) t^{\alpha -1}w\left( t\right) +\frac{4\pi }{\alpha
+1}t^{\alpha +1}  \label{a4} \\
&=&t^{\alpha }\int_{l\left( t\right) }\frac{\left\langle \nabla \left\vert
\nabla u\right\vert ,\nabla u\right\rangle }{\left\vert \nabla u\right\vert }%
-\left( \alpha +3\right) t^{\alpha -1}\int_{l\left( t\right) }\left\vert
\nabla u\right\vert ^{2}  \notag \\
&&+\frac{4\pi }{\alpha +1}t^{\alpha +1},  \notag
\end{eqnarray}%
where the last term should be replaced by $4\pi \ln t$ when $\alpha =-1.$

\begin{lemma}
\label{M}Let $\left( M^3,g\right) $ be a complete three-dimensional manifold
with finite first Betti number and finite number of ends. Assume that $u$ is a proper harmonic function 
on an end $E$ of $M$. Then%
\begin{equation*}
H_{\alpha }\left( T\right) \geq H_{\alpha }\left( s\right) -
\alpha \left( \alpha+2\right) \int_{s}^{T}w\left( t\right) t^{\alpha -2}dt+\frac{1}{2}%
\int_{L(s,T)}u^{\alpha }S|\nabla u|
\end{equation*}%
for all $s<T$. 
\end{lemma}

\begin{proof}
For regular values $s<T$ we have%
\begin{eqnarray*}
&&\left( T^{\alpha }\frac{dw}{dT}\left( T\right) -\alpha T^{\alpha
-1}w\left( T\right) \right) -\left( s^{\alpha }\frac{dw}{ds}(s)-\alpha
s^{\alpha -1}w\left( s\right) \right) \\
&=&\left( \int_{l\left( T\right) }u^{\alpha }\frac{\left\langle \nabla
\left\vert \nabla u\right\vert ,\nabla u\right\rangle }{\left\vert \nabla
u\right\vert }-\int_{l\left( s\right) }u^{\alpha }\frac{\left\langle \nabla
\left\vert \nabla u\right\vert ,\nabla u\right\rangle }{\left\vert \nabla
u\right\vert }\right) \\
&&-\left( \int_{l\left( T\right) }\left\vert \nabla u\right\vert \frac{%
\left\langle \nabla u^{\alpha },\nabla u\right\rangle }{\left\vert \nabla
u\right\vert }-\int_{l\left( s\right) }\left\vert \nabla u\right\vert \frac{%
\left\langle \nabla u^{\alpha },\nabla u\right\rangle }{\left\vert \nabla
u\right\vert }\right) \\
&=&\int_{L\left( s,T\right) }\left( u^{\alpha }\Delta \left\vert \nabla
u\right\vert -\left\vert \nabla u\right\vert \Delta u^{\alpha }\right) \\
&=&\int_{L\left( s,T\right) }\frac{u^{\alpha }}{\left\vert \nabla
u\right\vert }\left( \left\vert u_{ij}\right\vert ^{2}-\left\vert \nabla
\left\vert \nabla u\right\vert \right\vert ^{2}+\mathrm{Ric}\left( \nabla
u,\nabla u\right) \right) \\
&&-\alpha \left( \alpha -1\right) \int_{L\left( s,T\right) }\left\vert
\nabla u\right\vert ^{3}u^{\alpha -2}.
\end{eqnarray*}%
Note that the term $\frac{1}{\left\vert \nabla u\right\vert }\left\vert 
\mathrm{Ric}\left( \nabla u,\nabla u\right) \right\vert \leq \left\vert 
\mathrm{Ric}\right\vert \left\vert \nabla u\right\vert $ is integrable even
if $u$ has critical points in $L(s,T).$ The same can be concluded for $\frac{%
1}{\left\vert \nabla u\right\vert }\left( \left\vert u_{ij}\right\vert
^{2}-\left\vert \nabla \left\vert \nabla u\right\vert \right\vert
^{2}\right) $ via a classical regularization procedure \cite{LY1} from the
above identity by noticing that it is nonnegative due to the Kato inequality.

Using the co-area formula to rewrite the last term%
\begin{equation*}
\int_{L\left( s,T\right) }\left\vert \nabla u\right\vert ^{3}u^{\alpha
-2}=\int_{s}^{T}t^{\alpha -2}w\left( t\right) dt,
\end{equation*}%
we conclude%
\begin{eqnarray}
&&\left( T^{\alpha }\frac{dw}{dT}\left( T\right) -\alpha T^{\alpha
-1}w\left( T\right) \right) -\left( s^{\alpha }\frac{dw}{ds}(s)-\alpha
s^{\alpha -1}w\left( s\right) \right)  \label{a5} \\
&=&\int_{L\left( s,T\right) }\frac{u^{\alpha }}{\left\vert \nabla
u\right\vert }\left( \left\vert u_{ij}\right\vert ^{2}-\left\vert \nabla
\left\vert \nabla u\right\vert \right\vert ^{2}+\mathrm{Ric}\left( \nabla
u,\nabla u\right) \right)  \notag \\
&&-\alpha \left( \alpha -1\right) \int_{s}^{T}t^{\alpha -2}w\left( t\right)
dt.  \notag
\end{eqnarray}%
On the other hand, by (\ref{a1}),%
\begin{eqnarray}
&&\int_{l\left( t\right) }\frac{1}{\left\vert \nabla u\right\vert ^{2}}%
\left( \left\vert u_{ij}\right\vert ^{2}-\left\vert \nabla \left\vert \nabla
u\right\vert \right\vert ^{2}+\mathrm{Ric}\left( \nabla u,\nabla u\right)
\right)  \label{a7} \\
&=&\frac{1}{2}\int_{l\left( t\right) }\left( S-S_{t}+\frac{1}{\left\vert
\nabla u\right\vert ^{2}}\left\vert u_{ij}\right\vert ^{2}\right)  \notag
\end{eqnarray}%
on any regular level set $l\left( t\right) ,$ where $S_{t}$ is the scalar
curvature of $l\left( t\right) .$ Since by Lemma \ref{MV} the level set $%
l\left( t\right) $ is connected for all $t\geq 1,$ the Gauss-Bonnet theorem
and the Kato inequality imply that%
\begin{equation}
\frac{1}{2}\int_{l\left( t\right) }\left( S-S_{t}+\frac{1}{\left\vert \nabla
u\right\vert ^{2}}\left\vert u_{ij}\right\vert ^{2}\right) \geq \frac{1}{2}%
\int_{l(t)}S-4\pi +\frac{3}{4}\int_{l\left( t\right) }\frac{1}{\left\vert
\nabla u\right\vert ^{2}}\left\vert \nabla \left\vert \nabla u\right\vert
\right\vert ^{2}.  \label{a8}
\end{equation}%
Observe that 
\begin{equation*}
\left\vert w^{\prime }\left( t\right) \right\vert \leq \int_{l\left(
t\right) }\left\vert \nabla \left\vert \nabla u\right\vert \right\vert \leq
\left( \int_{l\left( t\right) }\frac{1}{\left\vert \nabla u\right\vert ^{2}}%
\left\vert \nabla \left\vert \nabla u\right\vert \right\vert ^{2}\right) ^{%
\frac{1}{2}}\left( \int_{l\left( t\right) }\left\vert \nabla u\right\vert
^{2}\right) ^{\frac{1}{2}},
\end{equation*}%
which implies%
\begin{equation*}
\int_{l\left( t\right) }\frac{1}{\left\vert \nabla u\right\vert ^{2}}%
\left\vert \nabla \left\vert \nabla u\right\vert \right\vert ^{2}\geq \frac{%
\left( w^{\prime }\right) ^{2}\left( t\right) }{w\left( t\right) }.
\end{equation*}%
Together with the elementary inequality%
\begin{equation*}
\frac{\left( w^{\prime }\right) ^{2}}{w}\left( t\right) \geq \frac{4}{t}%
w^{\prime }\left( t\right) -\frac{4}{t^{2}}w\left( t\right) ,
\end{equation*}%
one concludes from (\ref{a8}) that%
\begin{equation*}
\frac{1}{2}\int_{l\left( t\right) }\left( S-S_{t}+\frac{1}{\left\vert \nabla
u\right\vert ^{2}}\left\vert u_{ij}\right\vert ^{2}\right) \geq \frac{1}{2}%
\int_{l(t)}S-4\pi +\frac{3}{t}w^{\prime }\left( t\right) -\frac{3}{t^{2}}%
w\left( t\right) .
\end{equation*}%
Hence, by (\ref{a7}) we get 
\begin{eqnarray*}
&&\int_{l\left( t\right) }\frac{1}{\left\vert \nabla u\right\vert ^{2}}%
\left( \left\vert u_{ij}\right\vert ^{2}-\left\vert \nabla \left\vert \nabla
u\right\vert \right\vert ^{2}+\mathrm{Ric}\left( \nabla u,\nabla u\right)
\right) \\
&\geq &\frac{1}{2}\int_{l(t)}S-4\pi +\frac{3}{t}w^{\prime }\left( t\right) -%
\frac{3}{t^{2}}w\left( t\right)
\end{eqnarray*}%
for any regular level set $l\left( t\right) .$ We now use it and the co-area
formula to conclude%
\begin{eqnarray*}
&&\int_{L\left( s,T\right) }\frac{u^{\alpha }}{\left\vert \nabla
u\right\vert }\left( \left\vert u_{ij}\right\vert ^{2}-\left\vert \nabla
\left\vert \nabla u\right\vert \right\vert ^{2}+\mathrm{Ric}\left( \nabla
u,\nabla u\right) \right) \\
&\geq &\frac{1}{2}\int_{L(s,T)}u^{\alpha }S|\nabla u|+\int_{s}^{T}\left(
-4\pi +\frac{3}{t}w^{\prime }\left( t\right) -\frac{3}{t^{2}}w\left(
t\right) \right) t^{\alpha }dt.
\end{eqnarray*}%
Integrating by parts and noting that $w$ is Lipschitz, we obtain%
\begin{eqnarray*}
&&\int_{L\left( s,T\right) }\frac{u^{\alpha }}{\left\vert \nabla
u\right\vert }\left( \left\vert u_{ij}\right\vert ^{2}-\left\vert \nabla
\left\vert \nabla u\right\vert \right\vert ^{2}+\mathrm{Ric}\left( \nabla
u,\nabla u\right) \right) \\
&\geq &\frac{1}{2}\int_{L(s,T)}u^{\alpha }S|\nabla u|-\frac{4\pi }{\alpha +1}%
\left( T^{\alpha +1}-s^{\alpha +1}\right) \\
&+&3T^{\alpha -1}w\left( T\right) -3s^{\alpha -1}w\left( s\right) -3\alpha
\int_{s}^{T}t^{\alpha -2}w\left( t\right) dt.
\end{eqnarray*}%
Plugging this into (\ref{a5}) yields%
\begin{eqnarray*}
&&\left( T^{\alpha }\frac{dw}{dT}\left( T\right) -\alpha T^{\alpha
-1}w\left( T\right) \right) -\left( s^{\alpha }\frac{dw}{ds}(s)-\alpha
s^{\alpha -1}w\left( s\right) \right) \\
&\geq &\frac{1}{2}\int_{L(s,T)}u^{\alpha }S|\nabla u|-\frac{4\pi }{\alpha +1}%
\left( T^{\alpha +1}-s^{\alpha +1}\right) +3T^{\alpha -1}w\left( T\right)
-3s^{\alpha -1}w\left( s\right) \\
&&-\alpha \left( \alpha +2\right) \int_{s}^{T}t^{\alpha -2}w\left( t\right)
dt
\end{eqnarray*}%
for any $s<T.$ This proves the result.
\end{proof}

We are now ready to prove Theorem \ref{A5}.  We may assume that the
number of ends is one as otherwise $M$ must be a cylinder by the splitting
theorem \cite{CG}. Note also that for a complete
manifold $M$ with non-negative Ricci curvature, its first Betti number is
always bounded by its dimension \cite{And, Mi}.

Hence, without loss of generality, we may assume that all representatives of $H_1 (M)$ lie in some domain 
$D$, and $D\subset B_{p}(r_0)$ for
some $r_0$ large enough. We let $u$ denote the barrier function on $E=M\setminus D$, satisfying either (\ref{n}) or (\ref{p}).

If $M$ is nonparabolic, then the barrier function must be
proper as the minimal positive Green's function goes to $0$ at infinity 
\cite{LY}. 
Consequently, Lemma \ref{M} is applicable to the barrier
function $u$ of $M\setminus D.$

\begin{lemma}
\label{P} Let $(M^3,g)$ be a complete three-dimensional manifold with
non-negative Ricci curvature. If its scalar curvature is bounded below by $%
S(x)\geq \frac{c}{r(x)+1}$ on $M,$ then $M$ does not admit any positive
Green's function, i.e., $M$ is parabolic.
\end{lemma}

\begin{proof}
Suppose otherwise, that $M$ is nonparabolic. Apply Lemma \ref{M} to the barrier function $u$ of 
$M\setminus D$ with $\alpha =0.$ Then%
\begin{equation*}
H_{0}\left( T\right) \geq H_{0}\left( t\right) +\frac{1}{2}%
\int_{L(t,T)}\left\vert \nabla u\right\vert S
\end{equation*}%
for all $0<t<T\leq 1,$ where 
\begin{equation*}
H_{0}\left( t\right) =w^{\prime }\left( t\right) -3\frac{w\left( t\right) }{t%
}+4\pi t.
\end{equation*}
Note that $w\left( t\right) \leq Ct^{2}$ by the gradient estimate in \cite%
{CY0}, hence 
\begin{equation*}
\liminf_{t\rightarrow 0}H_{0}(t)=0.
\end{equation*}
Therefore, there exists a constant  $C>0$ so that
\begin{equation*}
\int_{M\setminus D}\left\vert \nabla u\right\vert
S=\int_{L(0,1)}\left\vert \nabla u\right\vert S\leq C.
\end{equation*}%
On the other hand, for any $R>r_0$,%
\begin{eqnarray*}
\int_{B_{p}\left( R\right) \setminus B_{p}(r_0)}\left\vert \nabla u\right\vert
S &=&\int_{r_0}^{R}\left( \int_{\partial B_{p}\left( r\right) }\left\vert
\nabla u\right\vert S\right) dr \\
&\geq &\int_{r_0}^{R}\frac{c}{r+1}\left( \int_{\partial B_{p}\left( r\right)
}\left\vert \nabla u\right\vert \right) dr \\
&\geq &\frac{1}{C}\ln R,
\end{eqnarray*}%
for some $C>0$, where we have used the fact that 
\begin{eqnarray*}
\int_{\partial B_{p}(r)}\left\vert \nabla u\right\vert &\geq& \left\vert
\int_{\partial B_{p}(r)}\frac{\partial u}{\partial r}\right\vert\\
&=& \left \vert\int_{\partial D}  \frac{\partial u}{\partial \nu}\right\vert>0.\\
\end{eqnarray*}
This contradiction completes the proof.
\end{proof}

The following corollary immediately follows from \cite{LY}.

\begin{corollary}
Let $(M^3,g)$ be a complete three-dimensional manifold with non-negative Ricci
curvature. If its scalar curvature is bounded below by 
$S(x)\geq \frac{c}{r(x)+1}$ on $M,$ then $\int_{1}^{\infty }\frac{t}{\mathrm{V}_{p}(t)}%
\,dt=\infty .$
\end{corollary}

For the proof of Theorem \ref{A5} we also need the following lemma, which is true in any dimension.  

\begin{lemma}
\label{Gr} Let $\left( M^{n},g\right) $ be a complete $n$-dimensional manifold with
non-negative Ricci curvature and $u$ a harmonic function on $M\setminus D$ such that 
\begin{equation*}
u=0 \text{ on }\partial D,\ \ \lim_{x\rightarrow \infty
}u(x)=\infty, \text{ and }\int_{\partial D}\frac{\partial u}{\partial \nu}=1,
\end{equation*}%
where $\nu$ denotes the outer normal to $D$. If 
\begin{equation*}
\mathrm{V}_{p}(R)\geq \beta \,R^{\alpha +1}\text{ \ for all \ }R\geq R_{0},
\end{equation*}%
where $\alpha \geq 0,$ $\beta >0$ and $R_{0}>0$ are all constants, then 
\begin{equation*}
\limsup_{x\rightarrow \infty }\left( \left\vert \nabla u\right\vert
\,r^{\alpha }\right) (x)\leq \frac{C}{\beta },
\end{equation*}%
for some constant $C$ depending only on dimension.
\end{lemma}

\begin{proof} Let $r_0>0$ be large enough so that $D\subset B_p(r_0)$. For $R\geq r_0,$ let 
\begin{equation*}
F(R)=\max_{x\in \partial B_{p}(R)}\left\vert \nabla u\right\vert (x).
\end{equation*}
We assert that there exists $C>0$ such that for any $r_0\leq R<T$,
\begin{equation}
\left\vert \nabla u\right\vert (x)-F(R)-\frac{C}{T}u(x)\leq 0
\label{a9}
\end{equation}%
on $B_{p}(T)\setminus B_{p}(R).$
Indeed, the function on the left hand side is subharmonic as
$\left\vert \nabla u\right\vert $ is subharmonic and $u$ is harmonic. Also,
it is nonpositive on the boundary of $B_{p}(T)\setminus B_{p}(R)$
by the gradient estimate of Cheng and Yau \cite{CY0}, which says that $\left\vert \nabla
u\right\vert (x)\leq \frac{C}{r\left( x\right) }u\left( x\right).$ So, (\ref{a9}) immediately follows from
the maximum principle. 

By first letting $T\rightarrow \infty $ and then maximizing
over all $x\in M\setminus B_p(R),$ one 
concludes from (\ref{a9}) that 
\begin{equation}
F(R)=\max_{x\in M\setminus B_{p}(R)}\left\vert \nabla u\right\vert (x)
\label{a10}
\end{equation}%
for any $R>r_0.$ In particular, $F(R)$ is non-increasing in $R.$ Let%
\begin{equation}\label{Gamma}
\Gamma =\limsup_{R\rightarrow \infty }\left( R^{\alpha }\,F(R)\right) .
\end{equation}
It suffices to show that $$\Gamma \leq \frac{C}{\beta },$$ for $C$ depending only on dimension. 

For $x\in \partial B_{p}\left( R\right) $ with $R\geq 2r_0,$ by the Bishop-Gromov
volume comparison theorem,%
\begin{equation*}
\frac{\mathrm{V}_{p}\left( R\right) }{\mathrm{V}_{x}\left( \frac{R}{2}%
\right) }\leq \frac{\mathrm{V}_{x}\left( 2R\right) }{\mathrm{V}_{x}\left( 
\frac{R}{2}\right) }\leq C.
\end{equation*}%
Therefore, as $\mathrm{V}_{p}(R)\geq \beta \,R^{\alpha +1}$ for $R\geq R_0$, it follows that
\begin{equation*}
{\mathrm{V}_{x}\left( \frac{R}{2}\right) }\geq \frac{\beta}{C} \,R^{\alpha +1}.
\end{equation*}%
We now follow an argument in \cite{NR}. Let
\begin{equation*}
a_x=\min_{B_{x}(\frac{R}{2})}u\,\ \text{and }b_x=\max_{B_{x}(\frac{R}{2})}u.
\end{equation*}%
Since $u$ is harmonic and $\int_{\partial D}\frac{\partial u}{\partial \nu}=1$, 
we have $\int_{l\left( t\right) }\left\vert \nabla
u\right\vert =1$. The co-area formula implies that 
\begin{equation*}
\int_{L\left( a,b\right) }\left\vert \nabla u\right\vert
^{2}=\int_{a_x}^{b_x}\int_{l\left( t\right) }\left\vert \nabla u\right\vert
=b_x-a_x\leq F\left( \frac{R}{2}\right) R,
\end{equation*}%
where the last inequality follows from (\ref{a10}). Applying the mean value
inequality to the subharmonic function $\left\vert \nabla u\right\vert ^{2}$
(see Theorem 7.2 in \cite{L}) we arrive at%
\begin{eqnarray*}
\left\vert \nabla u\right\vert ^{2}(x) &\leq &\frac{C}{\mathrm{V}_{x}\left( 
\frac{R}{2}\right) }\int_{B_{x}(\frac{R}{2})}\left\vert \nabla u\right\vert
^{2} \\
&\leq &\frac{C}{\beta R^{\alpha +1}}\int_{L(a_x,b_x)}\left\vert \nabla
u\right\vert ^{2} \\
&\leq &\frac{C}{\beta R^{\alpha }}F\left (\frac{R}{2}\right).
\end{eqnarray*}%
As $x\in \partial B_{p}\left( R\right) $ is arbitrary, this implies that $%
F^{2}(R)\leq \frac{C}{\beta R^{\alpha }}F\left(\frac{R}{2}\right )$. Rewrite it into
\begin{equation}
W^{2}(R)\leq \frac{C}{\beta }W\left(\frac{R}{2}\right),  \label{a11}
\end{equation}
where $W(R)=R^{\alpha }\,F(R)$. By a simple induction, we conclude $W(R)\leq
C(R_{0}).$ Hence, $\Gamma $ is finite. From (\ref{a11}), after letting $%
R\rightarrow \infty ,$ one obtains $\Gamma \leq \frac{C}{\beta }$ as desired.
By (\ref{Gamma}), this proves the lemma. 
\end{proof}
 For the reader's convenience, we restate Theorem \ref{A5} from the Introduction. 
\begin{theorem}
\label{A'5} Let $(M^3,g)$ be a complete three-dimensional manifold with
non-negative Ricci curvature.

\begin{itemize}
\item If the scalar curvature is bounded below by $S\geq 1$ on $M,$ then
there exists a universal constant $C>0$ such that 
\begin{equation}
\mathrm{V}_{p}\left( R\right) \leq C\,R  \label{part1}
\end{equation}%
for all $R>0$ and $p\in M.$

\item If the scalar curvature 
\begin{equation*}
S\left( x\right) \geq \frac{C_0}{r^{\alpha }\left( x\right) +1}
\end{equation*}%
for some $\alpha \in \left[ 0,1\right] ,$ where $r(x)$ is the geodesic
distance from $x$ to a fixed point $p\in M,$ then
\begin{equation}
\mathrm{V}_{p}\left( R_{i}\right) \leq \frac{C}{C_0}\,R_{i}^{\alpha +1}  \label{part2}
\end{equation}%
for a sequence $R_{i}\rightarrow \infty$, where $C>0$ is a universal constant.
\end{itemize}
\end{theorem}

\begin{proof}
By Lemma \ref{P}, $M$ must be parabolic, hence by (\ref{p}) it admits a
positive harmonic function $u$ on $M\setminus D$ such that $u=0$ on $\partial D$ and $
\lim_{x\rightarrow \infty }u(x)=\infty .$ Normalizing $u$ we have 
\begin{equation*}
u=0\text{ on }\partial D,\ \ \lim_{x\rightarrow \infty
}u(x)=\infty, \text{ and }\int_{\partial D}\frac{\partial u}{\partial \nu}=1.
\end{equation*}
According to \cite{Y2}, $\mathrm{V}_{p}(R)\geq \frac{1}{C}\,\mathrm{V}_{p}(1)\,R$, for
all $R\geq 1$, where $C>0$ is a universal constant. Applying Lemma \ref{Gr} for $\alpha =0$ and 
$\beta= \frac{\mathrm{V}_p (1)}{C}$,
it follows that 
\begin{equation}
\left\vert \nabla u\right\vert \leq \frac{C}{\mathrm{V}_{p}(1)}\text{ on }%
M\setminus B_{p}(r_{1})  \label{a12}
\end{equation}
for some $r_1>0$ sufficiently large. The monotonicity formula from Lemma \ref{M}
implies that 
\begin{equation}\label{S_Bound}
H_{0}\left( t\right) \geq H_{0}\left( 1\right) +\frac{1}{2}%
\int_{L(1,t)}\left\vert \nabla u\right\vert S
\end{equation}%
for all $t\geq 1,$ where 
\begin{equation*}
H_{0}\left( t\right) =w^{\prime }\left( t\right) -3\frac{w\left( t\right) }{t%
}+4\pi t
\end{equation*}%
and $w\left( t\right) =\int_{l\left( t\right) }\left\vert \nabla
u\right\vert ^{2}.$ 

Note that $w(t)$ is bounded, as $\int_{l\left( t\right) }\left\vert \nabla
u\right\vert =1$ and $\vert\nabla u\vert$ is bounded by (\ref{a12}). In
particular, there exists some constant  $C(r_1)>0$, depending on $\sup _{B_p (r_1)\setminus D} \vert \nabla u \vert$,  such that for each $t>1$ we have $w^{\prime
}(\xi )\leq C(r_1),$ for some $\xi \in \left( t,t+1\right) $. 

In conclusion, (\ref{S_Bound}) implies that for
all $t>1$,
\begin{equation}
\int_{L(1,t)}\left\vert \nabla u\right\vert S\leq \int_{L(1,\xi )}\left\vert
\nabla u\right\vert S\leq 8\pi \,t+C(r_1).  \label{a13}
\end{equation} 

To prove the first part (\ref{part1}), by scaling, it suffices to show there exists a universal constant $C>0$ so that  
\begin{equation}
\mathrm{V}_{p}(1)\leq \frac{C}{A}\text{ \ if \ }S\geq A>0.  \label{a14}
\end{equation}
According to (\ref{a12}), 
\begin{equation*}
u(x)\leq \frac{C}{\mathrm{V}_{p}(1)}\,r(x)+C(r_{1})
\end{equation*}
for all $x\in M\setminus B_{p}(r_{1})$. Hence, there exists a universal
constant $C>0$ so that 
\begin{equation*}
B_{p}(R)\setminus B_{p}(r_{1})\subset L\left( 1,\frac{C}{\mathrm{V}_{p}(1)}%
\,R+C(r_{1})\right)
\end{equation*}%
for all $R>r_1$, by additionally assuming that $r_{1}$ is large enough such that $u(x)\geq 1$ outside $B_{p}(r_{1}).$ From (%
\ref{a13}) and $S\geq A,$ we conclude
\begin{equation*}
A\,\int_{B_{p}(R)\setminus B_{p}(r_{1})}\left\vert \nabla u\right\vert \leq 
\frac{C}{\mathrm{V}_{p}(1)}\,R+C(r_{1})
\end{equation*}%
for all $R>r_{1}.$ However, 
\begin{equation*}
\int_{B_{p}(R)\setminus B_{p}(r_{1})}\left\vert \nabla u\right\vert \geq
\int_{r_{1}}^{R}\left( \,\int_{\partial B_{p}(r)}\frac{\partial u}{\partial r%
}\right) dr=R-r_{1}.
\end{equation*}%
Hence, $\mathrm{V}_{p}(1)\leq \frac{C}{A}$ after letting $R\rightarrow
\infty $. This proves (\ref{a14}).

Now we turn to the second part (\ref{part2}). 
Let us assume that  for some $\varepsilon>0$ we have 
\begin{equation} \label{a15}
\mathrm{V}_{p}\left( R\right) \geq \frac{1}{\varepsilon }\,R^{\alpha +1} , 
\end{equation}
 for all $R\geq R_{0}.$ Therefore, by Lemma \ref{Gr},%
\begin{equation}
\left\vert \nabla u\right\vert \leq C\varepsilon r^{-\alpha }\text{ \ on }%
M\setminus B_{p}(R_{1}),  \label{a16}
\end{equation}%
where $C$ is a universal constant and $R_{1}\geq R_{0}$ a fixed constant.
For $0\leq \alpha <1,$ integrating (\ref{a16}) along minimal geodesics we
obtain%
\begin{equation*}
u\leq \frac{C\,\varepsilon }{1-\alpha }r^{1-\alpha }+C(R_{1})\text{ \ on }%
M\setminus B_{p}(R_{1}).
\end{equation*}%
This implies that 
\begin{equation*}
B_{p}\left( R\right) \setminus B_{p}(R_{1})\subset L\left( 1,\frac{%
C\varepsilon }{1-\alpha }R^{1-\alpha }+C(R_{1})\right)
\end{equation*}%
by assuming additionally that $R_1$ is large enough so  such that $u(x)\geq 1$ outside $B_{p}(R_{1}).$
By (\ref{a13}), 
\begin{eqnarray}
\int_{B_{p}\left( R\right) \setminus B_{p}(R_{1})}S|\nabla u| &\leq
&\int_{L\left( 1,\frac{C\,\varepsilon }{1-\alpha }R^{1-\alpha
}+C(R_{1})\right) }S\left\vert \nabla u\right\vert  \label{a17} \\
&\leq &\frac{C\varepsilon }{1-\alpha }R^{1-\alpha }+C(R_{1}),  \notag
\end{eqnarray}%
where $C$ is a universal constant.  However, by the assumption that $S\geq
C_0\,r^{-\alpha }$ \ on $M\setminus B_{p}(R_1)$, we obtain from the co-area
formula that 
\begin{eqnarray}
\int_{B_{p}\left( R\right) \setminus B_{p}(R_{1})}S|\nabla u| &\geq
&\int_{R_{1}}^{R}\left( C_0\,r^{-\alpha }\right) \left( \int_{\partial
B_{p}\left( r\right) }\frac{\partial u}{\partial r}\right) dr  \label{a18} \\
&=&\frac{C_0}{1-\alpha }\left( R^{1-\alpha }-R_{1}^{1-\alpha }\right) .  \notag
\end{eqnarray}%
In conclusion, from (\ref{a17}) and (\ref{a18}),%
\begin{equation*}
\frac{C\,\varepsilon }{1-\alpha }R^{1-\alpha }+C(R_{1})\geq \frac{C_0}{%
1-\alpha }\left( R^{1-\alpha }-R_{1}^{1-\alpha }\right)
\end{equation*}%
for all $R>R_{1}$. Making $R\rightarrow \infty$ implies that 
$\varepsilon \geq \frac {C_0}{C}$. 
 Hence, (\ref{a15}) implies that
  \begin{equation*}
\mathrm{V}_{p}\left( R_{i}\right) \leq \frac{C}{C_0}\,R_{i}^{\alpha +1} , 
\end{equation*}
for a universal constant $C>0$. 
  A similar argument also works for $\alpha =1.$ This proves
the result.
\end{proof}

\section{Proof of Theorem \protect\ref{A2}\label{sect3'}}

In this section, we focus on the proof of Theorem \ref{A2}. First, we recall
a result of Bianchi and Setti (Theorem 2.1 in \cite{BS}) as the following lemma. 
While (\ref{b2}) is not explicitly stated by them, it can be easily verified from the construction.
Everywhere, $r(x)$ denotes the distance function to a fixed point $p\in M$. 
\begin{lemma} 
\label{eta}Let $\left( M^n,g\right) $ be an $n$-dimensional complete
Riemannian manifold with Ricci curvature bounded below by $\mathrm{Ric}\geq -%
\frac{C}{r^{2}+1}$. Then there exists a smooth proper function $\eta $ such
that%
\begin{eqnarray}
\frac{1}{C}\ln \left( r+2\right) &\leq &\eta \leq C\ln \left( r+2\right)
\label{b1} \\
\left\vert \nabla \eta \right\vert &\leq &\frac{C}{r+1}  \notag \\
\left\vert \Delta \eta \right\vert &\leq &\frac{C}{\left( r+1\right) ^{2}}\
\   \notag
\end{eqnarray}%
and 
\begin{equation}
\left\vert \nabla \left( \Delta \eta \right) \right\vert \leq \frac{C}{%
\left( r+1\right) ^{3}}+\frac{C}{r+1}\left\vert \nabla ^{2}\eta \right\vert
\label{b2}
\end{equation}%
for some constant $C>0.$
\end{lemma}

We henceforth denote with 
\begin{equation}
D\left( t\right) =\left\{ x\in M:\eta \left( x\right) <t\right\} .
\label{b3}
\end{equation}%
For given $R>0,$ let $\psi :\left( 0,\infty \right) \rightarrow \mathbb{R}$
be a smooth function such that $\psi =1$ on $\left( 0,\ln R\right) $ and $%
\psi =0$ on $\left( 2\ln R,\infty \right) ,$ and

\begin{equation*}
\left\vert \psi ^{\prime }\right\vert \leq \frac{C}{\ln R},\ \ \left\vert
\psi ^{\prime \prime }\right\vert \leq \frac{C}{\ln ^{2}R}\text{ and }%
\left\vert \psi ^{\prime \prime \prime }\right\vert \leq \frac{C}{\ln ^{3}R}.
\end{equation*}%
Composing it with $\eta ,$ we obtain a cut-off function $\psi \left(
x\right) =\psi \left( \eta \left( x\right) \right) $. Obviously,

\begin{equation*}
\psi =1\text{ on }D\left( \ln R\right) \text{ \ and }\psi =0\text{ on }%
M\setminus D\left( 2\ln R\right) .
\end{equation*}

In the following we assume that $\left( M,g\right) $ is nonparabolic and let 
$f$ be a barrier function on $M\setminus D$, see Definition \ref{barrier_defn}.  We extend $f$ to $M$ by setting 
$f=1$ on $D.$ The next result is well-known \cite{CY0}.

\begin{lemma}
\label{CY'} Let $\left( M^n,g\right) $ be an $n$-dimensional complete
noncompact Riemannian manifold with Ricci curvature bounded below by $%
\mathrm{Ric}\geq -\frac{C}{r^{2}+1}$. Then%
\begin{equation*}
\left\vert \nabla \ln f\right\vert \leq \frac{C}{r}\text{ \ on }M\setminus D.
\end{equation*}
\end{lemma}
We now consider the function $u$ given by 
\begin{equation}
u=f\psi .  \label{b4}
\end{equation}%
In view of Lemma \ref{eta}, it is straightforward to verify the following lemma.

\begin{lemma}
\label{u'}Let $\left( M^n,g\right) $ be an $n$-dimensional complete noncompact
Riemannian manifold with Ricci curvature bounded below by $\mathrm{Ric}\geq -%
\frac{C}{r^{2}+1}$. Then $u$ is harmonic on $D\left( \ln R\right) \setminus
D $ and on $M\setminus D,$%
\begin{eqnarray*}
\left\vert \nabla u\right\vert &\leq &u\,\left\vert \nabla \ln f\right\vert +%
\frac{C}{r}\frac{1}{\ln R} \\
\left\vert \Delta u\right\vert &\leq &\frac{C}{r^{2}\ln R}+\frac{C}{\ln R}%
\left\vert \nabla f\right\vert ^{2} \\
\left\vert \nabla \left( \Delta u\right) \right\vert &\leq &\frac{Cr}{\ln R}%
\left( \left\vert \nabla ^{2}\eta \right\vert ^{2}+\left\vert \nabla
^{2}f\right\vert ^{2}\right) +\frac{C}{r^{3}\ln R}+\frac{C}{\ln R}\left\vert
\nabla f\right\vert ^{2}.
\end{eqnarray*}
\end{lemma}

As mentioned before, in the case that $M$ has finite first Betti number and
finitely many, say $m, $ ends, $D$ is chosen to be large enough so that all
representatives of the first homology group $H_{1}\left( M\right) $ are
included in $D.$ Moreover, $M\setminus D$ has exactly $m$ unbounded
connected components. In this section, let $L\left( t,\infty \right) $
denote the connected component of the super-level set $\left\{ u>t\right\} $
that contains $D$ and
\begin{equation*}
l\left( t\right) =\partial L\left( t,\infty \right).
\end{equation*}

\begin{lemma}
\label{l'} Let $\left( M^n,g\right) $ be an $n$-dimensional complete
Riemannian manifold with $m$ ends and finite first Betti number $b_{1}\left(
M\right) $. Then for all $0<t<1,$%
\begin{equation*}
l\left( t\right) \cap \overline{M_{t}}=\partial M_{t}\text{ has }m\text{
connected components,}
\end{equation*}%
where $M_{t}$ is the union of all unbounded connected components of $%
M\setminus \overline{L\left( t,\infty \right) }.$
\end{lemma}

\begin{proof}
Recall that all representatives of the first homology $H_{1}\left( M\right) $
lie in $D$ and $M\setminus D$ has exactly $m$ unbounded components.
Therefore, as $D\subset L\left( t, \infty \right)$, it follows that  $L\left( t,\infty
\right) $ contains all representatives of $H_{1}\left( M\right)$ and $M_{t} $
has $m$ connected components for any $0<t<1.$ Note also $\partial
M_{t}\subset l\left( t\right).$

For fixed $0<\delta <1$ such that $t+\delta <1,$ define $U$ to be the union
of $L\left( t,\infty \right) $ with all bounded components of $M\setminus 
\overline{L\left( t+\delta ,\infty \right) }$ and $V$ the union of all
unbounded components of $M\setminus \overline{L\left( t+\delta ,\infty
\right) }.$ Since $M=U\cup V,$ we have the following Mayer-Vietoris sequence%
\begin{equation*}
H_{1}\left( U\right) \oplus H_{1}\left( V\right) \overset{j_{\ast }}{%
\rightarrow }H_{1}\left( M\right) \overset{\partial }{\rightarrow }%
H_{0}\left( U\cap V\right) \overset{i_{\ast }}{\rightarrow }H_{0}\left(
U\right) \oplus H_{0}\left( V\right) \overset{j_{\ast }^{\prime }}{%
\rightarrow }H_{0}(M).
\end{equation*}%
The map $j_{\ast }$ is onto because all representatives of $H_{1}\left(
M\right) $ lie inside $U$. The map $j_{\ast }^{\prime }$ is also onto. Note
also that $V$ has $m$ components and $U$ is connected. The latter is true
because each component of $M\setminus \overline{L\left( t+\delta ,\infty
\right) }$ intersects with $L\left( t,\infty \right) $ as $\overline{L\left(
t+\delta ,\infty \right) }\subset L\left( t,\infty \right) .$ We therefore
obtain the short exact sequence 
\begin{equation*}
0\rightarrow H_{0}\left( U\cap V\right) \rightarrow \mathbb{Z}\oplus
..\oplus \mathbb{Z\rightarrow Z\rightarrow }0
\end{equation*}%
with $m+1$ summands. In conclusion,%
\begin{equation*}
H_{0}\left( U\cap V\right) =\mathbb{Z}\oplus ..\oplus \mathbb{Z}
\end{equation*}%
with $m$ summands. Since $\delta >0$ can be arbitrarily small, this proves
that%
\begin{equation*}
l\left( t\right) \cap \overline{M_{t}}\text{ has }m\text{ components}
\end{equation*}%
for $0<t<1.$
\end{proof}

Recall that an $n$-dimensional manifold $M$ has asymptotically nonnegative
Ricci curvature if its Ricci curvature is bounded by

\begin{equation*}
\mathrm{Ric}(x)\geq -k\left( r(x)\right)
\end{equation*}%
for a continuous nonincreasing function $k:\left[ 0,\infty \right)
\rightarrow \left[ 0,\infty \right) $ with

\begin{equation*}
\int_{0}^{\infty }rk\left( r\right) dr<\infty .
\end{equation*}%
It is well known \cite{GW} that $\left( M^{n},g\right) $ has at most
Euclidean area growth, that is, the area $A_{p}(t)$ of geodesic sphere $%
\partial B_{p}(t)$ satisfies
\begin{equation}
\mathrm{A}_{p}\left( t\right) \leq C\,t^{n-1}  \label{b5}
\end{equation}%
for all $t>0.$ It is also clear that $\mathrm{Ric}\geq -\frac{C}{r^{2}+1}$.
In passing, we mention that Li and Tam \cite{LT1} have shown that $M^{n}$
has finitely many ends under the stronger assumption that $\int_{0}^{\infty
}r^{n-1}\,k\left( r\right) dr<\infty .$

We also note that according to Lemma \ref{eta}, 
\begin{equation}
D\left( 2\ln R\right) \subset B_{p}\left( R^{C}\right)  \label{b6}
\end{equation}%
for some $C>0$, where $D(t)$ was defined in (\ref{b3}).

We are now ready to prove Theorem \ref{A2} which is restated below.

\begin{theorem}
\label{T'} Let $\left( M^3,g\right) $ be a three-dimensional complete
noncompact Riemannian manifold with asymptotically nonnegative Ricci
curvature. Suppose that its scalar curvature is bounded below by a positive
constant and that $M$ has finitely many ends and finite first Betti number $%
b_{1}(M).$ Then $\left( M,g\right) $ is parabolic.
\end{theorem}

\begin{proof}
Assume by contradiction that $\left( M,g\right) $ is nonparabolic. Let $f$
be a barrier function on $M\setminus D$, see (\ref{n}), where the smooth
connected compact domain $D$ contains all representatives of $H_{1}\left(
M\right) $ and $M\setminus D$ has $m$ ends.

%Note that the assumption on scalar curvature implies% \begin{equation}
%S\geq \frac{1}{2}\frac{C_{0}}{r}\text{ \ on }M\setminus D.  \label{b7}
%\end{equation}

For given small $0<\varepsilon <1,$ consider $R$ large so that%
\begin{equation}
\ln R>\frac{1}{\varepsilon ^{2}}.  \label{b8}
\end{equation}%
In the following, the constant $C>0$ is independent of $R$ and $\varepsilon $,
but its value may change from line to line.

Define the function 
\begin{equation*}
u=f\psi
\end{equation*}%
as in (\ref{b4}). By the Bochner formula and the inequality $%
\left\langle \nabla \Delta u,\nabla u\right\rangle \geq -\left\vert \nabla
\left( \Delta u\right) \right\vert \left\vert \nabla u\right\vert $ it
follows that%
\begin{equation*}
\Delta \left\vert \nabla u\right\vert \geq \left( \left\vert
u_{ij}\right\vert ^{2}-\left\vert \nabla \left\vert \nabla u\right\vert
\right\vert ^{2}\right) \left\vert \nabla u\right\vert ^{-1}+\mathrm{Ric}%
\left( \nabla u,\nabla u\right) \left\vert \nabla u\right\vert
^{-1}-\left\vert \nabla \left( \Delta u\right) \right\vert
\end{equation*}%
holds on $M\setminus D$ whenever $\left\vert \nabla u\right\vert \neq 0.$

Let $\phi \left( x\right) $ be the smooth cut-off function given by%
\begin{equation*}
\phi (x)=\left\{ 
\begin{array}{c}
\phi \left( u\left( x\right) \right) \\ 
0%
\end{array}%
\right. 
\begin{array}{c}
\text{on }L\left( \varepsilon ,\infty \right) \\ 
\text{on }M\setminus L\left( \varepsilon ,\infty \right)%
\end{array}%
\end{equation*}%
Here, the function $\phi \left( t\right) $ is smooth, $\phi \left( t\right)
=1$ for $2\varepsilon \leq t\leq 1$ and $\phi \left( t\right) =0$ for $%
t<\varepsilon .$ Moreover,%
\begin{equation}
\left\vert \phi ^{\prime }\right\vert \left( t\right) \leq \frac{C}{%
\varepsilon }\text{ \ and }\left\vert \phi ^{\prime \prime }\right\vert
\left( t\right) \leq \frac{C}{\varepsilon ^{2}}\text{ \ for }\varepsilon
<t<2\varepsilon .  \label{b9}
\end{equation}
Clearly, the function $\phi \left( x\right) $ satisfies $\phi =1$ on $%
L\left( 2\varepsilon ,\infty \right) $ and $\phi =0$ on $M\setminus L\left(
\varepsilon ,\infty \right) .$ It follows that%
\begin{eqnarray}
&&\int_{M\setminus D}\left( \Delta \left\vert \nabla u\right\vert \right)
\phi ^{2}  \label{b10} \\
&\geq &\int_{M\setminus D}\left( \left\vert u_{ij}\right\vert
^{2}-\left\vert \nabla \left\vert \nabla u\right\vert \right\vert ^{2}+%
\mathrm{Ric}\left( \nabla u,\nabla u\right) \right) \left\vert \nabla
u\right\vert ^{-1}\phi ^{2}  \notag \\
&&-\int_{M\setminus D}\left\vert \nabla \left( \Delta u\right) \right\vert
\phi ^{2}.  \notag
\end{eqnarray}
To estimate the last term, we use Lemma \ref{u'} to obtain%
\begin{eqnarray}
\int_{M\setminus D}\left\vert \nabla \left( \Delta u\right) \right\vert \phi
^{2} &\leq &\frac{C}{\ln R}\int_{M\setminus D}\left\vert \nabla ^{2}\eta
\right\vert ^{2}r\phi ^{2}+\frac{C}{\ln R}\int_{M\setminus D}\left\vert
\nabla ^{2}f\right\vert ^{2}r\phi ^{2}  \label{b11} \\
&&+\frac{C}{\ln R}\int_{D\left( 2\ln R\right) \setminus D}\frac{1}{r^{3}}+%
\frac{C}{\ln R}\int_{M\setminus D}\left\vert \nabla f\right\vert ^{2}. 
\notag
\end{eqnarray}
By \cite{LT1}, $\int_{M}\left\vert \nabla f\right\vert ^{2}<\infty .$ By (%
\ref{b5}) and (\ref{b6}), 
\begin{equation}
\frac{C}{\ln R}\int_{D\left( 2\ln R\right) \setminus D}\frac{1}{r^{3}}\leq C.
\label{b12}
\end{equation}
Now we deal with the first term in (\ref{b11}). For convenience, we extend $%
\phi $ everywhere on $M$ by setting $\phi (x)=1$ for $x\in D$. Integrating
by parts gives 
\begin{eqnarray}
\int_{M}\left\vert \nabla ^{2}\eta \right\vert ^{2}\left( r+1\right) \phi
^{2} &=&-\int_{M}\left\langle \nabla \left( \Delta \eta \right) ,\nabla \eta
\right\rangle \left( r+1\right) \phi ^{2}  \label{b13} \\
&&-\int_{M}\mathrm{Ric}\left( \nabla \eta ,\nabla \eta \right) \left(
r+1\right) \phi ^{2}  \notag \\
&&-\int_{M}\eta _{ij}\eta _{i}r_{j}\phi ^{2}  \notag \\
&&-2\int_{M}\eta _{ij}\eta _{i}\phi _{j}(r+1)\phi .  \notag
\end{eqnarray}%
By Lemma \ref{eta} it follows that
\begin{eqnarray*}
-\int_{M}\left\langle \nabla \left( \Delta \eta \right) ,\nabla \eta
\right\rangle \left( r+1\right) \phi ^{2} &\leq &C\int_{M}\frac{1}{r+1}%
\left\vert \nabla ^{2}\eta \right\vert \phi ^{2}+C\int_{D\left( 2\ln
R\right) }\frac{1}{\left( r+1\right) ^{3}} \\
&\leq &\frac{1}{4}\int_{M}\left( r+1\right) \left\vert \nabla ^{2}\eta
\right\vert ^{2}\phi ^{2}+C\int_{D\left( 2\ln R\right) }\frac{1}{\left(
r+1\right) ^{3}} \\
&\leq &\frac{1}{4}\int_{M}\left( r+1\right) \left\vert \nabla ^{2}\eta
\right\vert ^{2}\phi ^{2}+C\ln R,
\end{eqnarray*}%
where in the last line we have used (\ref{b12}). According to the Ricci
lower bound, we have%
\begin{eqnarray*}
-\int_{M}\mathrm{Ric}\left( \nabla \eta ,\nabla \eta \right) \left(
r+1\right) \phi ^{2} &\leq &C\int_{M}\frac{1}{r+1}\left\vert \nabla \eta
\right\vert ^{2} \\
&\leq &C\int_{D\left( 2\ln R\right) }\frac{1}{\left( r+1\right) ^{3}} \\
&\leq &C\ln R.
\end{eqnarray*}%
Furthermore, we have 
\begin{eqnarray*}
-\int_{M}\eta _{ij}\eta _{i}r_{j}\phi ^{2} &\leq &\frac{1}{4}\int_{M}\left(
r+1\right) \left\vert \nabla ^{2}\eta \right\vert ^{2}\phi
^{2}+\int_{D\left( 2\ln R\right) }\frac{1}{\left( r+1\right) ^{3}} \\
&\leq &\frac{1}{4}\int_{M}\left( r+1\right) \left\vert \nabla ^{2}\eta
\right\vert ^{2}\phi ^{2}+C\ln R.
\end{eqnarray*}%
Finally, we estimate the last term in (\ref{b13}) by 
\begin{eqnarray*}
-2\int_{M}\eta _{ij}\eta _{i}\phi _{j}\phi (r+1) &\leq &\frac{1}{4}%
\int_{M}\left( r+1\right) \left\vert \nabla ^{2}\eta \right\vert ^{2}\phi
^{2}+4\int_{M}|\nabla \eta |^{2}|\nabla \phi |^{2}(r+1) \\
&\leq &\frac{1}{4}\int_{M}\left( r+1\right) \left\vert \nabla ^{2}\eta
\right\vert ^{2}\phi ^{2}+C\int_{M}\frac{1}{r+1}|\nabla \phi |^{2}.
\end{eqnarray*}%
Lemma \ref{u'} and Lemma \ref{CY'} imply that on $L(\varepsilon
,2\varepsilon )$, 
\begin{equation}
|\nabla u|\leq \frac{C\varepsilon }{r}.  \label{b14}
\end{equation}%
Using the definition of $\phi $ we see that $\left\vert \nabla \phi
\right\vert \leq \frac{C}{r}.$ Therefore,
\begin{equation*}
\int_{M}\frac{1}{r+1}|\nabla \phi |^{2}\leq \int_{D\left( 2\ln R\right) }%
\frac{1}{\left( r+1\right) ^{3}}\leq C\ln R.
\end{equation*}%
We have proved that 
\begin{equation*}
-2\int_{M}\eta _{ij}\eta _{i}\phi _{j}\phi (r+1)\leq \frac{1}{4}%
\int_{M}\left( r+1\right) \left\vert \nabla ^{2}\eta \right\vert ^{2}\phi
^{2}+C\ln R.
\end{equation*}%
Plugging these estimates into (\ref{b13}) implies that%
\begin{equation}
\frac{C}{\ln R}\int_{M}\left\vert \nabla ^{2}\eta \right\vert ^{2}\left(
r+1\right) \phi ^{2}\leq C.  \label{b15}
\end{equation}

Since $f$ is harmonic on $M\setminus D$, we similarly have 
\begin{eqnarray*}
\int_{M\setminus D}\left\vert \nabla ^{2}f\right\vert ^{2}r\phi ^{2}
&=&-\int_{M\setminus D}\mathrm{Ric}\left( \nabla f,\nabla f\right) r\phi
^{2}-\int_{M\setminus D}f_{ij}f_{i}r_{j}\phi ^{2} \\
&&-2\int_{M\setminus D}f_{ij}f_{i}\phi _{j}r\phi -\int_{\partial
D}f_{ij}f_{i}\nu _{j}r\phi ^{2} \\
&\leq &\frac{1}{2}\int_{M\setminus D}\left\vert \nabla ^{2}f\right\vert
^{2}r\phi ^{2}+C.
\end{eqnarray*}
This proves that
\begin{equation}
\frac{C}{\ln R}\int_{M\setminus D}\left\vert \nabla ^{2}f\right\vert
^{2}\left( r+1\right) \phi ^{2}\leq C.  \label{b16}
\end{equation}
In conclusion, plugging (\ref{b12}), (\ref{b15}), and (\ref{b16}) into (\ref{b11}) implies
that%
\begin{equation*}
\int_{M\setminus D}\left\vert \nabla \left( \Delta
u\right) \right\vert \phi ^{2}\leq C.
\end{equation*}
Consequently, (\ref{b10}) becomes%
\begin{eqnarray}
&&\int_{M\setminus D}\left( \Delta \left\vert \nabla u\right\vert \right)
\phi ^{2}  \label{b17} \\
&\geq &\int_{M\setminus D}\left( \left\vert u_{ij}\right\vert
^{2}-\left\vert \nabla \left\vert \nabla u\right\vert \right\vert ^{2}+%
\mathrm{Ric}\left( \nabla u,\nabla u\right) \right) \left\vert \nabla
u\right\vert ^{-1}\phi ^{2}-C.  \notag
\end{eqnarray}%
By the co-area formula, we have%
\begin{eqnarray*}
&&\int_{M\setminus D}\left( \left\vert u_{ij}\right\vert ^{2}-\left\vert
\nabla \left\vert \nabla u\right\vert \right\vert ^{2}+\mathrm{Ric}\left(
\nabla u,\nabla u\right) \right) \left\vert \nabla u\right\vert ^{-1}\phi
^{2} \\
&=&\int_{\varepsilon }^{1}\phi ^{2}\left( t\right) \int_{\left\{ u=t\right\}
}\left( \left\vert u_{ij}\right\vert ^{2}-\left\vert \nabla \left\vert
\nabla u\right\vert \right\vert ^{2}+\mathrm{Ric}\left( \nabla u,\nabla
u\right) \right) \left\vert \nabla u\right\vert ^{-2}dt \\
&=&\int_{\varepsilon }^{1}\phi ^{2}\left( t\right) \int_{l\left( t\right)
}\left( \left\vert u_{ij}\right\vert ^{2}-\left\vert \nabla \left\vert
\nabla u\right\vert \right\vert ^{2}+\mathrm{Ric}\left( \nabla u,\nabla
u\right) \right) \left\vert \nabla u\right\vert ^{-2}dt \\
&&+\int_{\varepsilon }^{1}\phi ^{2}\left( t\right) \int_{\widetilde{l}\left(
t\right) }\left( \left\vert u_{ij}\right\vert ^{2}-\left\vert \nabla
\left\vert \nabla u\right\vert \right\vert ^{2}+\mathrm{Ric}\left( \nabla
u,\nabla u\right) \right) \left\vert \nabla u\right\vert ^{-2}dt,
\end{eqnarray*}%
where 
\begin{equation}
\widetilde{l}\left( t\right) =\partial \widetilde{L}\left( t,\infty \right)
\label{b18}
\end{equation}%
and%
\begin{equation}
\widetilde{L}\left( t,\infty \right) =\left( \left\{ u>t\right\} \setminus
L\left( t,\infty \right) \right) \cap L\left( \varepsilon ,\infty \right) .
\label{b19}
\end{equation}
By the Kato inequality and noting that $\mathrm{Ric}\left( \nabla u,\nabla
u\right) \geq -\frac{C}{\left( r+1\right) ^{2}}\left\vert \nabla
u\right\vert ^{2}$ we have%
\begin{eqnarray*}
&&\int_{\varepsilon }^{1}\phi ^{2}\left( t\right) \int_{\widetilde{l}\left(
t\right) }\left( \left\vert u_{ij}\right\vert ^{2}-\left\vert \nabla
\left\vert \nabla u\right\vert \right\vert ^{2}+\mathrm{Ric}\left( \nabla
u,\nabla u\right) \right) \left\vert \nabla u\right\vert ^{-2}dt \\
&\geq &-C\int_{\varepsilon }^{1}\phi ^{2}\left( t\right) \int_{\left\{
u=t\right\} }\frac{1}{\left( r+1\right) ^{2}}dt \\
&\geq &-C\int_{L\left( \varepsilon ,1\right) }\frac{1}{\left( r+1\right) ^{2}%
}\left\vert \nabla u\right\vert ,
\end{eqnarray*}%
where the last line follows from the co-area formula. However,%
\begin{eqnarray}
\int_{M\setminus D}\frac{1}{r^{2}}\left\vert \nabla u\right\vert &\leq
&\int_{M\setminus D}\frac{1}{r^{2}}\left\vert \nabla f\right\vert +\frac{1}{%
\ln R}\int_{D(2\,\ln R)}\frac{1}{\left( r+1\right) ^{3}}  \label{b20} \\
&\leq &C.  \notag
\end{eqnarray}%
In conclusion, 
\begin{equation*}
\int_{\varepsilon }^{1}\phi ^{2}\left( t\right) \int_{\widetilde{l}\left(
t\right) }\left( \left\vert u_{ij}\right\vert ^{2}-\left\vert \nabla
\left\vert \nabla u\right\vert \right\vert ^{2}+\mathrm{Ric}\left( \nabla
u,\nabla u\right) \right) \left\vert \nabla u\right\vert ^{-2}dt\geq -C.
\end{equation*}%
Thus, we have that
\begin{eqnarray*}
&&\int_{M\setminus D}\left( \left\vert u_{ij}\right\vert ^{2}-\left\vert
\nabla \left\vert \nabla u\right\vert \right\vert ^{2}+\mathrm{Ric}\left(
\nabla u,\nabla u\right) \right) \left\vert \nabla u\right\vert ^{-1}\phi
^{2} \\
&\geq &\int_{\varepsilon }^{1}\phi ^{2}\left( t\right) \int_{l\left(
t\right) }\left( \left\vert u_{ij}\right\vert ^{2}-\left\vert \nabla
\left\vert \nabla u\right\vert \right\vert ^{2}+\mathrm{Ric}\left( \nabla
u,\nabla u\right) \right) \left\vert \nabla u\right\vert ^{-2}dt-C.
\end{eqnarray*}%
Rewrite it into
\begin{eqnarray}
&&\int_{M\setminus D}\left( \left\vert u_{ij}\right\vert ^{2}-\left\vert
\nabla \left\vert \nabla u\right\vert \right\vert ^{2}+\mathrm{Ric}\left(
\nabla u,\nabla u\right) \right) \left\vert \nabla u\right\vert ^{-1}\phi
^{2}  \label{b21} \\
&\geq &\int_{\varepsilon }^{1}\phi ^{2}\left( t\right) \int_{l\left( t\right)
\cap \overline{M_{t}}}\left( \left\vert u_{ij}\right\vert ^{2}-\left\vert
\nabla \left\vert \nabla u\right\vert \right\vert ^{2}+\mathrm{Ric}\left(
\nabla u,\nabla u\right) \right) \left\vert \nabla u\right\vert ^{-2}dt 
\notag \\
&&+\int_{\varepsilon }^{1}\phi ^{2}\left( t\right) \int_{l\left( t\right)
\cap \left( M\setminus \overline{M_{t}}\right) }\left( \left\vert
u_{ij}\right\vert ^{2}-\left\vert \nabla \left\vert \nabla u\right\vert
\right\vert ^{2}+\mathrm{Ric}\left( \nabla u,\nabla u\right) \right)
\left\vert \nabla u\right\vert ^{-2}dt -C.  \notag
\end{eqnarray}%
Since the Ricci curvature is bounded below by $\mathrm{Ric}\left( \nabla
u,\nabla u\right) \left\vert \nabla u\right\vert ^{-2}\geq -\frac{C}{\left(
r+1\right) ^{2}}$, together with the Kato inequality $\left\vert
u_{ij}\right\vert ^{2}\geq \left\vert \nabla \left\vert \nabla u\right\vert
\right\vert ^{2}$, it follows that%
\begin{eqnarray*}
&&\int_{\varepsilon }^{1}\phi ^{2}\left( t\right) \int_{l\left( t\right)
\cap \left( M\setminus \overline{M_{t}}\right) }\left( \left\vert
u_{ij}\right\vert ^{2}-\left\vert \nabla \left\vert \nabla u\right\vert
\right\vert ^{2}+\mathrm{Ric}\left( \nabla u,\nabla u\right) \right)
\left\vert \nabla u\right\vert ^{-2}dt \\
&\geq &-C\int_{\varepsilon }^{1}\int_{\left\{ u=t\right\} }\frac{1}{r^{2}}dt
\\
&\geq &-C\int_{M\setminus D}\frac{1}{r^{2}}\left\vert \nabla u\right\vert \\
&\geq &-C,
\end{eqnarray*}%
where the last line follows from (\ref{b20}). In conclusion, (\ref{b21})
becomes%
\begin{eqnarray}
&&\int_{M\setminus D}\left( \left\vert u_{ij}\right\vert ^{2}-\left\vert
\nabla \left\vert \nabla u\right\vert \right\vert ^{2}+\mathrm{Ric}\left(
\nabla u,\nabla u\right) \right) \left\vert \nabla u\right\vert ^{-1}\phi
^{2}  \label{b22} \\
&\geq &\int_{\varepsilon }^{1}\phi ^{2}\left( t\right) \int_{l\left(
t\right) \cap \overline{M_{t}}}\left( \left\vert u_{ij}\right\vert
^{2}-\left\vert \nabla \left\vert \nabla u\right\vert \right\vert ^{2}+%
\mathrm{Ric}\left( \nabla u,\nabla u\right) \right) \left\vert \nabla
u\right\vert ^{-2}dt-C.  \notag
\end{eqnarray}%
By (\ref{a1}) and the Kato inequality, 
\begin{equation}
\left( \left\vert u_{ij}\right\vert ^{2}-\left\vert \nabla \left\vert \nabla
u\right\vert \right\vert ^{2}+\mathrm{Ric}\left( \nabla u,\nabla u\right)
\right) \left\vert \nabla u\right\vert ^{-2}\geq \frac{1}{2}S-\frac{1}{2}
S_{t}.  \label{b23}
\end{equation}%
Note that by the Gauss-Bonnet theorem, in view of Lemma \ref{l'}, one has
\begin{equation*}
\int_{l\left( t\right) \cap \overline{M_{t}}}S_{t}\leq 8\pi m\leq C,
\end{equation*}%
for all regular values $t\in \left( 0,1\right)$. 
Therefore, we conclude from (\ref{b23}) that%
\begin{equation*}
\int_{l\left( t\right) \cap \overline{M_{t}}}\left( \left\vert
u_{ij}\right\vert ^{2}-\left\vert \nabla \left\vert \nabla u\right\vert
\right\vert ^{2}+\mathrm{Ric}\left( \nabla u,\nabla u\right) \right)
\left\vert \nabla u\right\vert ^{-2}\geq \frac{1}{2}\int_{l\left( t\right)
\cap \overline{M_{t}}}S-C.
\end{equation*}%
Together with (\ref{b22}), this implies that%
\begin{eqnarray}
&&\int_{M\setminus D}\left( \left\vert u_{ij}\right\vert ^{2}-\left\vert
\nabla \left\vert \nabla u\right\vert \right\vert ^{2}+\mathrm{Ric}\left(
\nabla u,\nabla u\right) \right) \left\vert \nabla u\right\vert ^{-1}\phi
^{2}  \label{b24} \\
&\geq &\frac{1}{2}\int_{\varepsilon }^{1}\phi ^{2}\left( t\right)
\int_{l\left( t\right) \cap \overline{M_{t}}}Sdt-C.  \notag
\end{eqnarray}
On the other hand, 
\begin{eqnarray*}
\int_{M\setminus D}\left( \Delta \left\vert \nabla u\right\vert \right) \phi
^{2} &=&\int_{M\setminus D}\left\vert \nabla u\right\vert \Delta \phi
^{2}-\int_{\partial D}\left\vert \nabla u\right\vert _{\nu } \\
&\leq &\frac{C}{\varepsilon }\int_{L\left( \varepsilon ,2\varepsilon \right)
}\left\vert \nabla u\right\vert \,\left\vert \Delta u\right\vert +\frac{C}{%
\varepsilon ^{2}}\int_{L\left( \varepsilon ,2\varepsilon \right) }\left\vert
\nabla u\right\vert ^{3}+C.
\end{eqnarray*}%
First, by (\ref{b14}) and Lemma \ref{u'} we have 
\begin{equation*}
\frac{C}{\varepsilon }\int_{L\left( \varepsilon ,2\varepsilon \right)
}\left\vert \nabla u\right\vert \,\left\vert \Delta u\right\vert \leq \frac{C%
}{\ln R}\int_{D\left( 2\ln R\right) \setminus D}\left( \frac{1}{r^{3}}%
+\left\vert \nabla f\right\vert ^{2}\right) \leq C,
\end{equation*}%
where the last inequality follows by (\ref{b12}). Similarly, we get 
\begin{equation*}
\frac{C}{\varepsilon ^{2}}\int_{L\left( \varepsilon ,2\varepsilon \right)
}\left\vert \nabla u\right\vert ^{3}\leq C\int_{D\left( 2\ln R\right)
\setminus D}\frac{1}{r^{2}}\left\vert \nabla u\right\vert \leq C,
\end{equation*}%
where the last inequality is by (\ref{b20}). In conclusion, 
\begin{equation*}
\int_{M\setminus D}\left( \Delta \left\vert \nabla u\right\vert \right) \phi
^{2}\leq C.
\end{equation*}%
Hence, (\ref{b17}) and (\ref{b24}) imply that 
\begin{equation*}
\int_{2\varepsilon }^{1}\left( \int_{l\left( t\right) \cap \overline{M_{t}}%
}S\right) dt\leq C.
\end{equation*}%
As $S$ is bounded from below by a positive constant and $l\left( t\right) \cap \overline{M_{t}}%
=\partial M_{t},$ it follows that 
\begin{equation}
\int_{2\varepsilon }^{1}\mathrm{A}\left( \partial M_{t}\right) dt\leq C
\label{b25}
\end{equation}%
for a constant $C>0$ independent of $\varepsilon >0$ and $R>0.$ 

Recall that
the barrier function defined by (\ref{n}) satisfies $\int_{\partial D}\left(
-\frac{\partial f}{\partial \nu }\right) =C>0$. Since $f$ is harmonic on the
bounded domain $M\setminus \left( M_{t}\cup D\right) $, it follows that 
\begin{equation}
C=-\int_{\partial D}\frac{\partial f}{\partial \nu }=-\int_{\partial M_{t}}%
\frac{\partial f}{\partial \nu }\leq \int_{\partial M_{t}}\left\vert \nabla
f\right\vert .  \label{b26}
\end{equation}%
According to Lemma \ref{CY'} we have 
\begin{equation}
\left\vert \nabla f\right\vert \leq \frac{C}{r}f\ \text{\ on }M\setminus D.
\label{b27}
\end{equation}%
If $x\in \partial M_{t}\cap D\left( \ln R\right),$ then $f\left( x\right)
=u\left( x\right) =t$ and we evidently get $\left\vert \nabla f\right\vert
\left( x\right) \leq Ct.$ If $x\in \partial M_{t}\setminus D\left( \ln
R\right),$ then (\ref{b27}) implies $\left\vert \nabla f\right\vert \left(
x\right) \leq \frac{C}{\ln R}\leq C \varepsilon^2$ in view of (\ref{b8}).
In conclusion,
\begin{equation*}
\left\vert \nabla f\right\vert \leq Ct\ \text{ on }\partial M_{t}\text{ for
all }t\in \left( 2\varepsilon ,1\right) .
\end{equation*}%
Plugging this into (\ref{b26}) one obtains that 
\begin{equation*}
\mathrm{A}\left( \partial M_{t}\right) \geq \frac{C}{t}\text{ \ for all }%
t\in \left( 2\varepsilon ,1\right) .
\end{equation*}%
However, this contradicts (\ref{b25}) if $\varepsilon $ is sufficiently
small. In conclusion, $\left( M,g\right) $ must be parabolic.
\end{proof}

\section{Proof of Theorem \protect\ref{A3} \label{sect4}}

We now turn to the proof of Theorem \ref{A3}.  Since $M$ is assumed to have finite first Betti number and 
finitely many ends, we may assume that $D\subset M$ is large enough so that all representatives $H_1(M)$ lie in $D$ and $M\setminus D$ has the maximal number of ends. 
The monotonicity
formula in Lemma \ref{M} gives the following integral estimate.

\begin{lemma}
\label{SL}Let $\left( M^3,g\right) $ be a parabolic complete three-dimensional
manifold with non-negative scalar curvature,  finite first Betti number, and finite number of ends. Assume
the barrier $u$ defined on an end $E$ of $M$ by (\ref{p}) has bounded gradient, $%
\sup_{E}\left\vert \nabla u\right\vert <\infty $. Then there exists a
constant $\Upsilon >0$ such that 
\begin{equation*}
\int_{L(1,t)}S|\nabla u|\leq 8\pi t+\Upsilon ,
\end{equation*}%
for all $t>1$.
\end{lemma}

\begin{proof}
Applying Lemma \ref{M} with $\alpha =0$ to $u$ we get%
\begin{equation}
\frac{1}{2}\int_{L(1,t)}S|\nabla u|\leq \frac{dw}{dt}-3\frac{w\left(
t\right) }{t}+4\pi t+\Upsilon _{0}  \label{c1}
\end{equation}%
for all $t>1,$ where%
\begin{equation*}
\Upsilon _{0}=-\left( \frac{dw}{dt} \big\vert _{t=1}-3w\left( 1\right) +4\pi \right) .
\end{equation*}
Since $u$ is harmonic, we know that $\int_{l\left( t\right) }\left\vert
\nabla u\right\vert $ is a constant. Together with the assumption that $u$
has bounded gradient, we obtain for all $t>1$
\begin{equation*}
w\left( t\right) =\int_{l\left( t\right) }\left\vert \nabla u\right\vert
^{2}\leq \Upsilon _{1},
\end{equation*}%
where $\Upsilon_1$ is a constant. So, for every $t>1$ there exists $\xi \in \left( t,t+1\right) $ such that $%
\frac{dw}{dt}\left( \xi \right) \leq \Upsilon _{1}$. Then (\ref{c1}) implies%
\begin{equation*}
\int_{L(1,t)}S|\nabla u|\leq \int_{L(1,\xi )}S|\nabla u|\leq 8\pi t+\Upsilon,
\end{equation*}%
where $\Upsilon =2\Upsilon _{0}+2\Upsilon _{1}+8\pi .$ This proves the
result.
\end{proof}

In the following, we again normalize the barrier function $u$ on the end $E$ such that 
\begin{equation}
\int_{l\left( t\right) }\left\vert \nabla u\right\vert =1  \label{c2}
\end{equation}%
for all $t\geq 1$. Assume that $D\subset B_p (r_0)$, for some $r_0>0$ large enough. 

\begin{lemma}
\label{MR}Let $\left( M^3,g\right) $ be a parabolic complete three-dimensional
Riemannian manifold. Assume that the Ricci curvature of $M$ is bounded below
by $\mathrm{Ric}\geq -k\left( r\left( x\right) \right) $, for a continuous
nonincreasing function $k\left( r\right) $ satisfying $\int_{0}^{\infty
}rk\left( r\right) dr<\infty .$ Normalize the barrier $u$ of an end $E$ as in (\ref{c2}).
Then
\begin{equation*}
M\left( R\right) =\sup_{E\cap\left(B_{p}\left( 2R\right) \setminus B_{p}\left( R\right)\right)
}\left\vert \nabla u\right\vert
\end{equation*}
satisfies 
\begin{equation*}
M^{2}\left( R\right) \leq CR \left( \inf_{x\in E\cap\left(B_{p}\left( 2R\right)
\setminus B_{p}\left( R\right) \right)} \mathrm{V}_{x}\left( \frac{R}{2}\right)
\right) ^{-1} \left( M\left( \frac{R}{2}\right) +M\left( R\right) +M\left(
2R\right) \right)
\end{equation*}
for a constant $C$ depending only on $k$ and for all $R>2r_{0}$.
\end{lemma}

\begin{proof}
We follow an argument in \cite{NR}. Pick $x\in E\cap \partial B_{p}\left( t\right) 
$, where $R\leq t \leq 2R $, such that $M\left( R\right)
=\left\vert \nabla u\right\vert \left( x\right) $. On $B_{x}\left( \frac{R}{2%
}\right) $ we have $\mathrm{Ric}\geq -\frac{C}{R^{2}}$ and $\Delta
\left\vert \nabla u\right\vert \geq -\frac{C}{R^{2}}\left\vert \nabla
u\right\vert $. Applying the mean value inequality from \cite{LT3} we get%
\begin{equation*}
\left\vert \nabla u\right\vert ^{2}\left( x\right) \leq \frac{C}{\mathrm{V}%
_{x}\left( \frac{R}{2}\right) }\int_{B_{x}\left( \frac{R}{2}\right)
}\left\vert \nabla u\right\vert ^{2},
\end{equation*}%
for a constant $C$ independent of $R$. 
This shows that%
\begin{equation}
M^{2}\left( R\right) \leq \frac{C}{\mathrm{V}_{x}\left( \frac{R}{2}\right) }%
\int_{L\left( \alpha _{x},\beta _{x}\right) }\left\vert \nabla u\right\vert
^{2},  \label{c3}
\end{equation}%
where%
\begin{equation*}
\alpha _{x}=\min_{B_{x}\left( \frac{R}{2}\right) }u\text{ and }\beta
_{x}=\max_{B_{x}\left( \frac{R}{2}\right) }u.
\end{equation*}%
By the co-area formula and (\ref{c2}) we have 
\begin{equation*}
\int_{L\left( \alpha _{x},\beta _{x}\right) }\left\vert \nabla u\right\vert
^{2}=\int_{\alpha _{x}}^{\beta _{x}}\int_{l\left( t\right) }\left\vert
\nabla u\right\vert dt=\beta _{x}-\alpha _{x}\leq R\sup_{B_{x}\left( \frac{R%
}{2}\right) }\left\vert \nabla u\right\vert .
\end{equation*}%
Since $B_x\left(\frac{R}{2}\right) \subset B_{p}\left( \frac{3R}{2}\right) \setminus B_{p}\left( \frac{R}{2}\right)$, it
follows that%
\begin{equation*}
\int_{L\left( \alpha _{x},\beta _{x}\right) }\left\vert \nabla u\right\vert
^{2}\leq R\left( M\left( \frac{R}{2}\right) +M\left( R\right) +M\left(
2R\right) \right) .
\end{equation*}%
Together with (\ref{c3}) we conclude that%
\begin{equation*}
M^{2}\left( R\right) \leq \frac{CR}{\mathrm{V}_{x}\left( \frac{R}{2}\right) }%
\left( M\left( \frac{R}{2}\right) +M\left( R\right) +M\left( 2R\right)
\right)
\end{equation*}%
for some $x\in E\cap\left( B_{p}\left( 2R\right) \setminus B_{p}\left( R\right)\right) $ and
all $R>2r_{0}.$ This proves the result.
\end{proof}

We are now ready to prove Theorem \ref{A3}.

\begin{theorem}
\label{T2}Let $\left( M^3,g\right) $ be a three-dimensional complete
noncompact Riemannian manifold with finitely many ends and finite first
Betti number $b_{1}(M)<\infty .$ Assume that the Ricci curvature of $M$ is
bounded from below by $\mathrm{Ric}\left( x\right) \geq -k\left( r\left(
x\right) \right) $ for a continuous non-increasing function $k\left(
r\right) $ satisfying $\int_{0}^{\infty }rk\left( r\right) dr<\infty .$ Then
there exists a constant $C_{0}>0,$ depending only on $k,$ such that 
\begin{equation*}
\liminf_{x\rightarrow \infty }S\left( x\right) \leq \frac{C_{0}}{\mathrm{V}%
_{p}\left( 1\right) },
\end{equation*}%
where $\mathrm{V}_{p}\left( 1\right) $ is the volume of the geodesic ball $%
B_{p}\left( 1\right) .$
\end{theorem}

\begin{proof}
Let us assume by contradiction that 
\begin{equation}
\liminf_{x\rightarrow \infty }S\left( x\right) >\frac{\Gamma }{\mathrm{V}%
_{p}\left( 1\right) },  \label{c4}
\end{equation}%
where $\Gamma >0$ is a large enough constant to be specified later. In
particular,%
\begin{equation}
S\geq \frac{1}{2}\frac{\Gamma }{\mathrm{V}_{p}\left( 1\right) }\text{ \ on }%
M\setminus B_{p}\left( R_{0}\right)   \label{c5}
\end{equation}%
for some large enough $R_{0}.$ Below, $C$ denotes a constant that depends
only on $k.$

By Theorem \ref{T'} we know that $M$ is parabolic. By localizing a barrier
function to an end as above, we may assume without loss of generality that $M
$ has only one end. According to Nakai \cite{N}, there exists a proper
harmonic function $u:M\setminus D\rightarrow \left( 0,\infty \right) $ with $%
u=0$ on $\partial D.$ We normalize $u$ as in (\ref{c2}).

Since $M$ has asymptotically nonnegative Ricci curvature, a volume
comparison result in \cite{LT2} implies%
\begin{equation*}
\mathrm{V}_{x}\left( \frac{R}{2}\right) \geq \frac{1}{C}R\mathrm{V}%
_{p}\left( 1\right) 
\end{equation*}%
for any $x\in \partial B_{p}\left( R\right) .$ Choosing $R_{0}>0$ large enough so that  $D\subset B_{p}(R_{0})$ and applying Lemma %
\ref{MR} we get%
\begin{equation}
M^{2}\left( R\right) \leq \frac{C_1}{\mathrm{V}_{p}\left( 1\right) }\left(
M\left( \frac{R}{2}\right) +M\left( R\right) +M\left( 2R\right) \right) 
\label{c6}
\end{equation}%
for all $R>2R_{0}$, where $C_1$ is a constant depending only on the function $k$.  

On the other hand, by Lemma \ref{CY'} we know that $\left\vert
\nabla \ln u\right\vert \leq \frac{C}{r}$ \ on $M\setminus B_{p}\left(
2R_{0}\right) $. It follows that
\begin{equation*}
u\leq r^{C}\sup_{\partial B_{p}\left( 2R_{0}\right) }u\text{ \ \ on }%
M\setminus B_{p}\left( 2R_{0}\right) 
\end{equation*}%
and furthermore that
\begin{equation*}
\left\vert \nabla u\right\vert \leq r^{C}\sup_{\partial B_{p}\left(
2R_{0}\right) }u\text{ \ on }M\setminus B_{p}\left( 2R_{0}\right) .
\end{equation*}%
In particular, there exists $C_{2}$ with $C_{2}>2C_{1}$, where $C_1$ is the constant in (\ref{c6}),  and there exists 
$R_{1}>2R_{0}$ sufficiently large, such that
\begin{equation}
M\left( R\right) \leq \frac{1}{\mathrm{V}_{p}\left( 1\right) }R^{\frac{C_{2}%
}{2}}  \label{c7}
\end{equation}%
for all $R\geq R_{1}.$

We now claim that for $m\geq 1,$%
\begin{equation}
M\left( R\right) \leq \frac{2^{C_{2}}}{\mathrm{V}_{p}\left( 1\right) }R^{%
\frac{C_{2}}{2^{m}}}  \label{c8}
\end{equation}%
for any $R\geq 2^{m}R_{1}$.

By (\ref{c7}) we see that (\ref{c8}) holds for $m=1.$ We assume by induction
that it holds for $m$ and prove it for $m+1.$ In other words, we aim to
prove that%
\begin{equation}
M\left( R\right) \leq \frac{2^{C_{2}}}{\mathrm{V}_{p}\left( 1\right) }R^{%
\frac{C_{2}}{2^{m+1}}}  \label{c9}
\end{equation}%
for all $R\geq 2^{m+1}R_{1}.$ By the induction hypothesis (\ref{c8}) we get that%
\begin{eqnarray*}
M\left( \frac{R}{2}\right)  &\leq &\frac{2^{C_{2}}}{\mathrm{V}_{p}\left(
1\right) }\left( \frac{R}{2}\right) ^{\frac{C_{2}}{2^{m}}}, \\
M\left( R\right)  &\leq &\frac{2^{C_{2}}}{\mathrm{V}_{p}\left( 1\right) }R^{%
\frac{C_{2}}{2^{m}}}, \\
M\left( 2R\right)  &\leq &\frac{2^{C_{2}}}{\mathrm{V}_{p}\left( 1\right) }%
\left( 2R\right) ^{\frac{C_{2}}{2^{m}}}.
\end{eqnarray*}

Using (\ref{c6}) we therefore obtain%
\begin{eqnarray*}
M^{2}\left( R\right)  &\leq &\frac{2^{C_{2}}C_{1}}{\left( \mathrm{V}%
_{p}\left( 1\right) \right) ^{2}}\left( \left( \frac{1}{2}\right) ^{\frac{%
C_{2}}{2^{m}}}+1+2^{\frac{C_{2}}{2^{m}}}\right) R^{\frac{C_{2}}{2^{m}}} \\
&\leq &\frac{2^{2C_{2}}}{\left( \mathrm{V}_{p}\left( 1\right) \right) ^{2}}%
R^{\frac{C_{2}}{2^{m}}},
\end{eqnarray*}%
where the second line follows from $C_{2}>2C_{1}.$ In conclusion,%
\begin{equation*}
M\left( R\right) \leq \frac{2^{C_{2}}}{\mathrm{V}_{p}\left( 1\right) }R^{%
\frac{C_{2}}{2^{m+1}}}
\end{equation*}%
and (\ref{c9}) is proved. Hence, (\ref{c8}) holds for any $m\geq 1$ and $%
R\geq 2^{m}R_{1}.$ By taking $2^{m}=\left[ \ln R\right] ,$ this readily
implies that%
\begin{equation*}
M\left( R\right) \leq \frac{C}{\mathrm{V}_{p}\left( 1\right) }
\end{equation*}%
for all $R\geq R_{1}^{2}.$ In other words,
\begin{equation}
\sup_{M\setminus B_{p}\left( R_{2}\right) }\left\vert \nabla u\right\vert
\leq \frac{C}{\mathrm{V}_{p}\left( 1\right) }  \label{c10}
\end{equation}
for sufficiently large $R_{2}>R_{1}$, where the constant $C$ depends only on the function $k$.   It in turn implies that%
\begin{equation}
u\left( x\right) \leq \frac{C}{\mathrm{V}_{p}\left( 1\right) }r\left(
x\right) +C(R_{2})\text{ \ on }M\setminus B_{p}\left( R_{2}\right) 
\label{c11}
\end{equation}%
for some constant $C(R_{2})$ depending on $R_{2}.$ By (\ref{c11}) we
conclude that 
\begin{equation*}
B_{p}\left( R\right) \setminus B_{p}\left( R_{2}\right) \subset L\left( 1,%
\frac{C}{\mathrm{V}_{p}\left( 1\right) }R+C(R_{2})\right) .
\end{equation*}
According to Lemma \ref{SL}
we have
\begin{eqnarray}
\int_{B_{p}\left( R\right) \setminus B_{p}\left( R_{2}\right) }S\left\vert
\nabla u\right\vert  &\leq &\int_{L\left( 1,\frac{C}{\mathrm{V}_{p}\left(
1\right) }R+C(R_{2})\right) }S\left\vert \nabla u\right\vert   \label{c12} \\
&\leq &\frac{C}{\mathrm{V}_{p}\left( 1\right) }R+C(R_{2})+\Upsilon   \notag
\end{eqnarray}%
for all $R>R_{2},$ where $\Upsilon $ is a constant.

Now by (\ref{c5}) it follows that $S\geq \frac{1}{2}\frac{\Gamma }{\mathrm{V}%
_{p}\left( 1\right) }$ \ on $B_{p}\left( R\right) \setminus B_{p}\left(
R_{2}\right) $, whereas by (\ref{c2}) we have $\int_{\partial B_{p}\left(
r\right) }\frac{\partial u}{\partial r}=1$, for all $\ r>R_{2}.$ The co-area
formula then implies%
\begin{eqnarray*}
\int_{B_{p}\left( R\right) \setminus B_{p}\left( R_{2}\right) }S\left\vert
\nabla u\right\vert  &\geq &\frac{1}{2}\frac{\Gamma }{\mathrm{V}_{p}\left(
1\right) }\int_{B_{p}\left( R\right) \setminus B_{p}\left( R_{2}\right)
}\left\vert \nabla u\right\vert  \\
&\geq &\frac{1}{2}\frac{\Gamma }{\mathrm{V}_{p}\left( 1\right) }%
\int_{R_{2}}^{R}\left( \int_{\partial B_{p}\left( r\right) }\frac{\partial u%
}{\partial r}\right) dr \\
&\geq &\frac{1}{2}\frac{\Gamma }{\mathrm{V}_{p}\left( 1\right) }\left(
R-R_{2}\right) .
\end{eqnarray*}%
From (\ref{c12}) we get 
\begin{equation*}
\frac{\Gamma -C}{\mathrm{V}_{p}\left( 1\right) }R\leq C(R_{2})+\Upsilon .
\end{equation*}%
Since $R>R_{2}$ is arbitrary, this forces $\Gamma \leq C$ and the theorem is
proved.
\end{proof}

\end{document}